\documentclass[a4paper,reqno]{amsart}

\usepackage{amsmath, amssymb, amsthm, graphicx, xcolor, hyperref, mathtools, bm, mdframed, framed, tikz, nicematrix, enumitem}

\newtheorem{theorem}{Theorem}[section]
\newtheorem{theoremLetter}{Theorem}[section]

\newtheorem{proposition}[theorem]{Proposition}
\newtheorem{lemma}[theorem]{Lemma}
\newtheorem{definition}[theorem]{Definition}
\newtheorem{remark}[theorem]{Remark}
\newtheorem{notation}[theorem]{Notation}

\newcommand{\Z}{\mathbb{Z}}
\newcommand{\N}{\mathbb{N}}

\DeclareMathOperator{\End}{End}
\DeclareMathOperator{\Id}{Id}
\DeclareMathOperator{\Image}{Im}

\title{Twisted conjugacy growth series of virtually abelian groups}

\author[A.~Evetts]{Alex Evetts}
\address{The University of Manchester, Manchester, UK, and Heilbronn Institute for Mathematical Research, Bristol, UK}
\email{alex.evetts@manchester.ac.uk}

\author[M.~Lathouwers]{Maarten Lathouwers}
\address{KU Leuven Kulak Kortrijk Campus, Kortrijk, Belgium, and funded by FWO PhD-fellowship fundamental research (file number: 1102424N).}
\email{maarten.lathouwers@kuleuven.be}

\date{}

\begin{document}

\begin{abstract}
    We initiate the study of the \emph{twisted conjugacy growth series} of a finitely generated group, the formal power series associated to the twisted conjugacy growth function. Our main result is that, for a virtually abelian group, this series is always an explicitly computable $\N$-rational function. As a corollary, we obtain a similar result for the relative growth series of a twisted conjugacy class in a virtually abelian group.
\end{abstract}

\maketitle

\section{Introduction}
In finitely generated groups, it is long established to study the \emph{(word) growth} of the group, i.e. the function that counts how many group elements can be represented by words of a fixed length in the generators (and their inverses). A standard reference is \cite{Mann}. In 1988, Babenko \cite{Babenko} was the first to introduce the concept of \emph{conjugacy growth} where one is interested in conjugacy classes that are represented by words of a fixed length. The conjugacy growth function has been studied extensively (see for example \cite{GubaSapir} for an overview). As with all growth functions, one can associate a formal power series, called the \emph{conjugacy growth series}, to the conjugacy growth. In \cite{AntolinCiobanu} and \cite{CHHR}, it was shown that a hyperbolic group has a rational conjugacy growth series if and only if it is virtually cyclic, which confirmed a conjecture of Rivin. Other work on conjugacy growth includes \cite{CEH}, \cite{Heisenberg}, \cite{GY}, \cite{Greenfeld}.

Conjugacy in a group $G$ can be generalised by considering an endomorphism $\varphi\in \End(G)$. More precisely, we say that two elements $a,b\in G$ are \emph{$\varphi$-twisted conjugate} (also known as \emph{twisted conjugate}) if there exists some $c\in G$ such that $a=cb\varphi(c)^{-1}$. This induces an equivalence relation on the group $G$ and partitions it into \emph{$\varphi$-twisted conjugacy classes} (or \emph{twisted conjugacy classes}). A group $G$ has the \emph{$R_\infty$-property} if all automorphisms of $G$ give rise to an infinite number of twisted conjugacy classes. This property has particularly nice connections with fixed point theory (see \cite{HeathKeppelmann}, \cite{Jiang}) and has been extensively studied (see for example \cite{DekimpeGoncalves}, \cite{Nasybullov}). Instead of only wondering if a morphism induces an infinite number of twisted conjugacy classes, the \emph{twisted conjugacy growth} studies how the number of twisted conjugacy classes (with a representative of increasing length) behaves.\\

In \cite{Benson}, Benson studied the growth of finitely generated virtually abelian groups by proving that the standard (weighted) growth series is rational. His techniques, including for example the use of polyhedral sets, were used multiple times in order to prove rationality results for virtually abelian groups. In \cite{Evetts}, the first named author generalised Benson's techniques to show that the (weighted) conjugacy growth series and the (weighted) coset growth series are rational for all virtually abelian groups. In Section \ref{sec:coset}, we make the arguments of the latter constructive. For more applications of polyhedral sets to virtually abelian groups, see \cite{Bishop}, \cite{CE}, \cite{CEL}. In \cite{DL}, the second named author began the systematic study of twisted conjugacy growth by considering finitely generated virtually abelian groups. The current paper continues the extensive work on virtually abelian groups by proving that the formal power series that is associated to the twisted conjugacy growth, namely the \emph{twisted conjugacy growth series}, is always rational, and can moreover be explicitly computed.

The \emph{$\N$-rational} functions are a proper subset of the rational functions (see Definition \ref{def:Nrational}). We say that such a function is \emph{explicitly computable} to mean that there exists an algorithm taking appropriate inputs that outputs the required function (see Definition \ref{def:expcomp}).

\begin{theoremLetter}\label{thmLetter:mainthm}
    Let $G$ be a finitely generated virtually abelian group with a finite weighted (monoid) generating set $\Sigma$ and $\varphi\in \End(G)$. Then the $\varphi$-twisted conjugacy growth series of $G$ with respect to $\Sigma$ is an $\N$-rational function. Moreover, given a description for $G$, the $\N$-rational function is explicitly computable.
\end{theoremLetter}

In Section \ref{sec:def} we state the definitions of the different types of growth and their associated formal power series; we explain the theory of (effectively constructible) polyhedral sets and we define the input data for finitely generated virtually abelian groups. In Section \ref{sec:coset}, we repeat the construction from \cite{Evetts} to obtain geodesic coset representatives and argue that this can be done in an effective way. In Section \ref{sec:main}, we describe the twisted conjugacy classes in virtually abelian groups and use this in combination with the work from \cite{Benson} and from Section \ref{sec:coset} to prove Theorem \ref{thmLetter:mainthm}. We conclude the paper with a consequence of our description of twisted conjugacy classes, namely that the growth series counting just the elements of any fixed twisted conjugacy class (the relative growth of the twisted conjugacy class) is rational.
\begin{theoremLetter}
    Let $G$ be a finitely generated virtually abelian group with a finite weighted (monoid) generating set $\Sigma$ and $\varphi\in \End(G)$. Then the relative growth series of every $\varphi$-twisted conjugacy class in $G$ with respect to $\Sigma$ is an $\N$-rational function. Moreover, given a description for $G$, the $\N$-rational function is explicitly computable.
\end{theoremLetter}

Throughout the paper, we use $\N$ for the non-negative integers and $\N_+$ for the positive integers.

\section{Definitions and basic results}\label{sec:def}
\subsection{Growth}
For a group $G$ and an endomorphism $\varphi\in \End(G)$ we say that two elements $a,b\in G$ are \emph{$\varphi$-twisted conjugate} (or \emph{twisted conjugate}), denoted $a\sim_\varphi b$, if there exists some $c\in G$ such that $a=cb\varphi(c)^{-1}$. This induces an equivalence relation $\sim_\varphi$ on $G$. Its equivalence classes $[a]_\varphi$ are called \emph{$\varphi$-twisted conjugacy classes} (or \emph{twisted conjugacy classes}). Note that if $\varphi=\Id$ is the identity morphism, we obtain the usual notion of conjugacy in the group $G$. Twisted conjugacy is indeed a generalisation of regular conjugacy.\\

Let $G$ be moreover finitely generated by some set $\Sigma$. Throughout the paper, we assume that $\Sigma$ generates $G$ as a monoid, i.e. that we take only positive powers of $\Sigma$. Denote by $\Sigma^\ast$ the words over the alphabet $\Sigma$ and by $\overline{w}\in G$ the evaluation of some word $w\in \Sigma^\ast$ in the group $G$. We call a function $\omega_\Sigma:\Sigma\to \N_+$ a \emph{weight function} on $\Sigma$ and we say that $\Sigma$ is a \emph{finite weighted generating set} of $G$. It induces a weight function on the elements of $G$ by setting for $g\in G$
\[\omega_\Sigma(g)=\begin{cases}
    0, &\text{if } g=e\\
    \min\left\{\sum\limits_{i=1}^r \omega_\Sigma(s_i) : s_1,\ldots,s_r\in \Sigma,\: \overline{s_1\cdots s_r}=g\right\}, &\text{else.}
\end{cases}\]
We call a word $s_1\cdots s_r\in \Sigma^\ast$ a \emph{geodesic representative} of $g\in G$ if $\overline{s_1\cdots s_r}=g$ and  $\omega_\Sigma(g)=\sum_i \omega_\Sigma(s_i)$. If $\omega_\Sigma(s)= 1$ for all $s\in\Sigma$, then the induced weight function on $G$ coincides with the classical word length $|\cdot|_\Sigma$ of an element of $G$ with respect to $\Sigma$. 
\begin{definition}[Growth functions]\label{def:growth functions}
    Let $\Sigma$ be a finite weighted generating set of $G$. Then we define the weight $\omega_\Sigma(U)$ of a subset $U\subset G$ by
    \[\omega_\Sigma(U)=\min\{\omega_\Sigma(u) : u\in U\}.\]
    For an endomorphism $\varphi\in \End(G)$ we define
    \begin{enumerate}
        \item the \emph{(weighted) standard growth function} $\beta_{G,\omega_\Sigma}$ by
            \[\beta_{G,\omega_\Sigma}:\N\to \N: r\mapsto \#\{g\in G : \omega_\Sigma(g)\leq r\}.\]
        \item the \emph{(weighted) ($\varphi$-)twisted conjugacy growth function} $c_{G,\omega_\Sigma}^\varphi$ by
            \[c_{G,\omega_\Sigma}^\varphi:\N\to \N: r\mapsto \#\{[g]_\varphi\in G/\sim_\varphi : \omega_\Sigma([g]_\varphi)\leq r\}.\]
        \item the \emph{(weighted) relative growth function} $\beta_{U\subset G,\omega_\Sigma}$ of $U$ in $G$ by
            \[\beta_{U\subset G,\omega_\Sigma}:\N\to \N: r\mapsto \#\{g\in U: \omega_\Sigma(g)\leq r\}.\]
    \end{enumerate}
\end{definition}

To these growth functions we associate formal power series $B_{G,\omega_\Sigma}$, $C_{G,\omega_\Sigma}^\varphi$ and $B_{U\subset G,\omega_\Sigma}$ by setting, for example,
\[B_{G,\omega_\Sigma}(z)=\sum_{r=0}^\infty \beta_{G,\omega_\Sigma}(r)z^r\]
and we call them the (weighted) \emph{standard growth series}, \emph{($\varphi$-)twisted conjugacy growth series} and \emph{relative growth series} of $U$ in $G$, respectively.\\

\subsection{Polyhedral sets}
We make essential use of the notion of a \emph{polyhedral set}, as introduced by Benson and used in \cite{Benson}, \cite{CEL}, \cite{Evetts}, and others.
\begin{definition}[Polyhedral sets]\label{def:PolSets}
	Let $k\in\N_+$, and let $\cdot$ denote the Euclidean scalar product on $\mathbb{R}^k$.
	\begin{enumerate}
		\item[(i)] A subset of $\Z^k$ of the form
		\begin{enumerate}
			\item[(1)] $\{z\in\Z^k : u\cdot z=a\}$,
			\item[(2)] $\{z\in\Z^k : u\cdot z\equiv a\mod b\}$, or
			\item[(3)] $\{z\in\Z^k : u\cdot z>a\}$
		\end{enumerate}
		for constants $u\in\Z^k$, $a\in\Z$, $b\in\N_+$, is called an \emph{elementary set};
		\item[(ii)]  a finite intersection of elementary sets is called a \emph{basic polyhedral set};
		\item[(iii)]  a finite union of basic polyhedral sets is called a \emph{polyhedral set}.
	\end{enumerate}
\end{definition}
It turns out that the notion of a polyhedral set is just one way to define a very useful class of subsets of free abelian groups. The class of polyhedral sets coincides with \emph{Presburger definable sets}, \emph{semilinear subsets} of $\Z^k$, and \emph{rational subsets} of $\Z^k$ (for definitions see \cite{ES} and \cite{GS}). For a discussion of the equivalences, and more references, see \cite{CE} and \cite{CEL}. The current paper is naturally a generalisation of \cite{Benson} and \cite{Evetts}, so we use the terminology of polyhedral sets for consistency.

\begin{definition}[Effectively constructible]
    A polyhedral set $X$ is called \emph{effectively constructible}, given some input data, if there exists an algorithm that takes the input data and produces a description for $X$ as a finite union of finite intersections of elementary sets.
\end{definition}
For the purposes of the current paper, the input data in question may be descriptions of other polyhedral sets, explicit maps between polyhedral sets, or ultimately the description of a virtually abelian group defined in Definition \ref{def:input data}. Note that we do not make any claims regarding the complexity of such an algorithm.
\begin{definition}[Integral affine transformation]
    A map $\phi\colon \Z^k\to\Z^l$, for some positive integers $k,l$, is an \emph{integral affine transformation} if there is an integer-valued $l\times k$ matrix $A$ and a constant vector $b\in\Z^l$ such that $\phi(x)=Ax+b$ for all $x\in\Z^k$.
\end{definition}

We will make use of the following basic results from the literature, proofs of which can be found in \cite{Benson} or \cite{Silva}.
\begin{proposition}\label{prop:polyhedral sets basic properties}
    If $X$ is a polyhedral subset of $\Z^k$ for some $k\in \N_+$, then the following holds.
	\begin{enumerate}
		\item[$(i)$] If $Y$ is a polyhedral subset of $\Z^k$, then the sets $X\cap Y$, $X\cup Y$, $X\setminus Y$ are polyhedral, and moreover effectively constructible, given the descriptions of $X$ and $Y$.
		\item[$(ii)$] If $Z\subset\Z^l$ is polyhedral for some $l\in \N_+$, then the product $X\times Z$ is a polyhedral subset of $\Z^{k+l}$ that is effectively constructible from the descriptions of $X$ and $Z$.
		\item[$(iii)$] If $\phi\colon \Z^k\to\Z^l$ and $\psi\colon\Z^l\to\Z^k$ are integral affine transformations, then the image $\phi(X)\subset\Z^l$ and preimage $\psi^{-1}(X)\subset\Z^l$ are polyhedral sets, and moreover effectively constructible, given the descriptions of $X$, $\phi$, $\psi$.\\
		In particular, the projection of a polyhedral set $X\subset \Z^k$ onto some subset of the coordinates of $\Z^k$ is an effectively constructible polyhedral set.
	\end{enumerate}
\end{proposition}
The key feature of polyhedral sets for our purposes is that their growth series, with respect to the $\ell_1$ distance, are rational functions. 
\begin{theorem}[See \cite{Benson}]\label{thm:polyrational}
    Let $X\subset\N^k$, for $k\in \N_+$, be a polyhedral set endowed with the $\ell_1$ metric. Then its growth series
    \[\sum_{n=0}^\infty \#\{g\in X : |g|_{\ell_1}\leq n\}z^n\]
    is a rational function.
\end{theorem}
In fact, the growth series turn out to be a particularly restricted form of rational function, which arises in formal language theory. See \cite{Nrational} and the references therein, and \cite[Section II]{SS}.
\begin{definition}[$\N$-rational functions]\label{def:Nrational}
   The set of \emph{$\N$-rational functions} is the smallest class $R$ of generating functions in one variable $z$, satisfying
    \begin{enumerate}
        \item $0$ and $z$ are in $R$,
        \item if $f(z)$ and $g(z)$ are in $R$, then so too are $f(z)+g(z)$ and $f(z)g(z)$,
        \item if $h(z)\in R$ is such that $h(0)=0$, then $\frac{1}{1-h(z)}\in R$.
    \end{enumerate}
\end{definition}

\begin{definition}[Explicitly computable rational function]\label{def:expcomp}
    We will say that a rational function $p(z)/q(z)$ is \emph{explicitly computable}, with respect to some description, if there exists an algorithm taking the description as input and outputting the two polynomials $p(z)$ and $q(z)$.
\end{definition}
Again, we make no claims regarding the complexity of the algorithm, merely that it exists. With these definitions, Theorem \ref{thm:polyrational} can be refined as follows.
\begin{proposition}[See {\cite[Proposition 4.14]{CEL}}]\label{prop:Nrational}
    Let $X\subset\Z^k$, for $k\in \N_+$, be a polyhedral set endowed with the $\ell_1$ metric. Then its growth series is $\N$-rational and can be explicitly computed from its description as a polyhedral set.
\end{proposition}

\subsection{Describing virtually abelian groups}
This article asserts that the twisted conjugacy growth series of a finitely generated virtually abelian group can be explicitly computed. For this to make sense, we need to clearly define the input data.
\begin{remark}\label{fully characteristic Zn}
    Every finitely generated (infinite) virtually abelian group $G$ has a finite index normal subgroup isomorphic to $\Z^n$ (for some $n\in \N_+$) that is fully characteristic, i.e. this subgroup is invariant under every endomorphism of $G$ (see for example \cite[Remark 3.3]{DL}).
\end{remark}

 Following \cite{CEL}, we define a standard normal form.
\begin{definition}[Subgroup-transversal normal form]
    Let $G$ be a finitely generated virtually abelian group with fully characteristic finite index normal subgroup $\Z^n$. Choose a basis $B=\{e_1,\ldots,e_n\}$ for the subgroup $\Z^n$, and a transversal $T$ for its (right) cosets in $G$. The \emph{subgroup-transversal normal form} for $G$ is the set of words $(\{e_1\}^*\cup \{e_1^{-1}\}^*)\cdots(\{e_n\}^*\cup \{e_n^{-1}\}^*)T$.
\end{definition}
For a finite set $S$, we write $S^{\pm1}$ to denote the union of $S$ with the set of all inverses of elements of $S$.
\begin{definition}[Description of a virtually abelian group]\label{def:input data}
    Let $G$ be a finitely generated (infinite) virtually abelian group, and $\varphi$ some endomorphism of $G$. A \emph{description} of $G$ and $\varphi$ consists of
    \begin{enumerate}
        \item a fully characteristic finite index normal subgroup $\Z^n$ together with a $\Z$-basis $B=\{e_1,\ldots,e_n\}$,
        \item a choice of transversal $T\subset G$ for the cosets of $\Z^n$ (where we assume that $1_G\in T$),
        \item a function $f\colon (B^{\pm1}\cup T)\times (B^{\pm1}\cup T)\to (B^{\pm1})^*T$ that maps a pair $(a,b)$ to the subgroup-transversal normal form of the product $\overline{ab}$,
        \item the images $\varphi(g)$, for each $g\in B^{\pm1}\cup T$, given in subgroup-transversal normal form,
        \item a finite set $\Sigma$ of generators for $G$ (generating $G$ as a monoid), each given as a word in $(B^{\pm1}\cup T)^*$,
        \item a weight function $\omega_\Sigma\colon \Sigma\to\N_+$.
    \end{enumerate}
\end{definition}

\section{Effective construction of geodesic coset representatives}\label{sec:coset}
For a finitely generated virtually abelian group $G$, with any choice of finite generating set, Benson \cite{Benson} proved that there exists a set of unique geodesic representatives for the elements of $G$, which can be counted using finitely many polyhedral sets, implying that the standard growth series of $G$ is rational. In \cite{CEL}, it is shown that the proof can be made effective, and so the growth series can be explicitly calculated. In \cite{Evetts}, the first named author generalised Benson's result to obtain unique geodesic representatives for the cosets of any subgroup of $G$, yielding rational coset growth series. In this section, we make this latter proof effective (see Theorem \ref{thm:cosetreps} and its proof), which is a key step in proving our main Theorem \ref{thmLetter:mainthm}. We closely follow \cite[Section 4]{CEL}.

\begin{definition}[Extended finite generating set]\label{def:extendedgenset}
    Let $G$ be virtually abelian, with a finite weighted generating set $\Sigma$ and fully characteristic finite index normal subgroup $\Z^n$ for some $n\in \N_+$.
    \begin{enumerate}
        \item Define
    \[S=\{\overline{s_1\cdots s_r} : s_i\in\Sigma, 1\leq r\leq [G:\Z^n]\}.\] 
    We call $S$ the \emph{extended finite generating set} of $G$ with respect to $\Sigma$. 
        \item If $\omega_\Sigma\colon\Sigma\to\N_+$ is the weight function on $\Sigma$, define a weight function on $S$ by $\omega_S(s)=\omega_\Sigma(s)$, for each $s\in S$.
    \end{enumerate}
\end{definition}

\begin{remark}
    The extended generating set $S$ is also a finite (monoid) generating set for $G$, and
    \[\omega_\Sigma(g)=\min\{\omega_\Sigma(v): v\in\Sigma^*, \overline{v}=g\} = \min\{\omega_S(u): u\in S^*, \overline{u}=g\} = \omega_S(g)\]
    for every $g\in G$. Thus, changing the generating set has not changed the weight of an element. Since there is no ambiguity, we use $\omega(g)$ for the weight of $g\in G$ from now on.
\end{remark}

\begin{lemma}[See {\cite[Remark 11.2]{Benson}}]\label{lem:pigeon}
    Let $G$ be finitely generated virtually abelian with $\Z^n$ as a fully characteristic finite index normal subgroup. Then every (non-empty) product $g_1\cdots g_r$ of elements of $G$ with $r\geq[G:\Z^n]$ contains a subproduct $g_ig_{i+1}\cdots g_{i+j}\in \Z^n$.
\end{lemma}
\begin{proof}
    Consider the prefix subproducts $g_1, g_1g_2, \ldots, g_1\cdots g_r$. If $r>[G:\Z^n]$ then the pigeonhole principle ensures that two such subproducts represent elements of the same $\Z^n$-coset, say $g_1g_2\cdots g_{i-1}$ and $g_1g_2\cdots g_{i+j}$ for some $i,j$. Then
    \[\left(g_1g_2\cdots g_{i-1}\right)^{-1}g_1g_2\cdots g_{i+j} = g_ig_{i+1}\cdots g_{i+j}\in\Z^n.\]
    In the case $r=[G:\Z^n]$, either there is a pair of subproducts in the same coset, or else each subproduct is in a different coset, and in particular one of them is contained in the identity coset $\Z^n$.
\end{proof}

\begin{definition}[Patterns]\label{def:patterns}
    Let $G$ be a virtually abelian group, $S$ an extended finite generating set of $G$ as in Definition \ref{def:extendedgenset} and $\Z^n$ a fully characteristic finite index normal subgroup.
    \begin{enumerate}
        \item Let $X=S\cap\Z^n=\{x_1,\ldots,x_r\}$ and $Y=S\setminus X=\{y_1,\ldots,y_s\}$. Define $P=\{\pi\in Y^* : |\pi|_S\leq[G:\Z^n]\}$ to be the (finite) set of \emph{patterns}, that is, all words in $Y^*$ of length at most the index $[G:\Z^n]$.
        \item For each $\pi=y_{i_1}\cdots y_{i_k}\in P$ with $|\pi|_S=k$, define the set $W^\pi$ of $\pi$-patterned words in $S^*$ as
        \[\left\{x_1^{w_1}\cdots x_r^{w_r} y_{i_1} x_1^{w_{r+1}}\cdots x_r^{w_{2r}}y_{i_2} x_1^{w_{2r+1}}\cdots y_{i_k} x_1^{w_{kr+1}}\cdots x_r^{w_{kr+r}} : \begin{array}{l}
            w_i\in\N\text{, for} \\
            1\leq i\leq kr+r
        \end{array}\right\}.\]
        \item For each $\pi\in P$ with $|\pi|_S=k$, define the bijection $\phi_\pi\colon W^\pi\to\N^{kr+r}$ recording the powers of the generators $x_i$
        \[\phi_\pi\colon x_1^{w_1}\cdots x_r^{w_r}y_{i_1} x_1^{w_{r+1}}\cdots x_r^{w_{2r}}y_{i_2} x_1^{w_{2r+1}}\cdots y_{i_k} x_1^{w_{kr+1}}\cdots x_r^{w_{kr+r}} \mapsto \begin{pmatrix} w_1 \\ w_2 \\ \vdots \\ w_{kr+r} \end{pmatrix}.\]
        We write $m_\pi = kr+r = (|\pi|_S+1)r$.
    \end{enumerate}
\end{definition}
By the following proposition, we only need to consider the finitely many patterns in $P$ and their corresponding sets of patterned words to represent geodesics in our group.
\begin{proposition}[See {\cite[Proposition 11.3]{Benson}}]\label{prop:P patterns}
    For each element $g\in G$, there exists some $\pi\in P$ such that $W^\pi$ contains a geodesic representative for $g$.
\end{proposition}

A non-effective version of the next result appears in the proof of \cite[Theorem 4.2]{Evetts}. In the rest of this section, we explain how to make the result effective and hence derive the following result.
\begin{theorem}\label{thm:cosetreps}
    Let $G$ be a virtually abelian group with finite generating set $\Sigma$ and a fully characteristic finite index normal subgroup $\Z^n$. Let $H$ be a subgroup contained in $\Z^n$ (given by a basis written in terms of the standard basis elements of $\Z^n$). Then for each pattern $\pi\in P$, there exists a set of geodesics $U^{\pi}_H\subset W^\pi\subset \Sigma^*$ such that $\phi_\pi(U^{\pi}_H)\subset\N^{m_\pi}$ is an effectively constructible polyhedral set, and the disjoint union $\bigcup_{\pi\in P}U^{\pi}_H$ comprises a complete set of unique representatives for the right cosets $G/H$.
\end{theorem}

\begin{remark}
    If $w\in W^\pi$, then since $\Z^n$ is normal in $G$, the coset $\Z^n\overline{w}$ represented by the element $\overline{w}$ depends only on the pattern $\pi$. 
\end{remark}

\begin{notation}
    Recall from Definition \ref{def:input data} that we use $T$ to denote a fixed transversal for the cosets of $\Z^n$ in $G$. For each $\pi\in P$, define $t_\pi\in T$ by requiring that
\[\overline{\pi}\in\Z^nt_\pi.\]
\end{notation}

We now demonstrate how to pass between a patterned word and the subgroup-transversal normal form of the element it represents. The following proposition is a restatement of \cite[Definition 2.10]{Evetts}. We include a proof for completeness, and to emphasise the fact that the process is effective.
\begin{notation}
    For $h\in G$ we denote by
    \[\tau_h:G\to G:g\mapsto hgh^{-1}\]
    the inner automorphism of $G$ given by conjugation by $h$.\\
    Recall that we fixed a basis $\{e_1,\dots,e_n\}$ of $\Z^n$ as input data (see Definition \ref{def:input data}). We identify every element of $\Z^n$ with its coordinates with respect to this basis. Moreover, we use additive notation in $\Z^n$ in parallel to the multiplicative notation used for $G$. For example, for $e_1^{a_1}\dots e_n^{a_n}, e_1^{b_1}\dots e_n^{b_n}\in \Z^n$ and $t\in T$, we use
    \[\left(\begin{pmatrix}a_1\\\vdots\\a_n\end{pmatrix}+\begin{pmatrix}b_1\\\vdots\\b_n\end{pmatrix}\right)t \quad \text{instead of} \quad e_1^{a_1}\dots e_n^{a_n}e_1^{b_1}\dots e_n^{b_n}t.\]
\end{notation}
\begin{proposition}\label{prop:representing words using matrices}
    Given a description of a virtually abelian group $G$ (as in Definition \ref{def:input data}), fix a pattern $\pi\in P$ and a coset representative $t\in T$. Then there are $s\in T$, vectors $A_{i,t}^\pi\in\Z^{m_\pi}$ and integers $B_{i,t}^\pi\in\Z$ (for $1\leq i \leq n$) such that for each $w\in W^\pi$ we have
    \begin{equation*}
        \tau_t(\overline{w}) = \left( \begin{pmatrix} A_{1,t}^\pi\cdot\phi_\pi(w) \\ \vdots \\ A_{n,t}^\pi\cdot\phi_\pi(w) \end{pmatrix} + \begin{pmatrix} B_{1,t}^\pi \\ \vdots \\ B_{n,t}^\pi \end{pmatrix} \right) s
        \quad\text{and}\quad
        \tau_t(\overline{\pi}) = \begin{pmatrix} B_{1,t}^\pi \\ \vdots \\ B_{n,t}^\pi \end{pmatrix} s.
    \end{equation*}
    
    Furthermore, we can find a vector $A_{n+1}^\pi\in\Z^n$ and integer $B_{n+1}^\pi\in\Z$ such that $\omega(w)=A_{n+1}^\pi\cdot\phi_\pi(w) + B_{n+1}^\pi$ for every $w\in W^\pi$.
\end{proposition}
\begin{proof}
    Since $\tau_t$ (for each $t\in T$) is an automorphism of $\Z^n$ that we can explicitly describe via an integer-valued $n\times n$ matrix $T_t$ using the description of $G$, we obtain that $\tau_t(z)=T_tz$ for $z\in\Z^n$. Partitioning the generating set into $X$ and $Y$ as in Definition \ref{def:patterns}, we note that conjugation by $y\in Y$ also defines an automorphism of $\Z^n$, and we can again determine the corresponding matrix $T_y$ so that $\tau_y(z)=T_yz$, for every $z\in\Z^n$. Now write $\Delta$ for the $n\times r$ matrix whose columns are the representations of the elements of $X=\{x_1,\ldots,x_r\}\subset \Z^n$ in terms of the basis $\{e_1,\ldots,e_n\}$ of $\Z^n$, so that $x_1^{v_1}x_2^{v_2}\cdots x_r^{v_r} = \Delta\begin{pmatrix} v_1\\ \vdots\\v_r \end{pmatrix}\in\Z^n$ for every $v_i\in\Z$.

    Now for a $\pi$-patterned word $w\in W^{\pi}$, we have 
    \begin{align*}
        w = x_1^{w_1}\cdots x_r^{w_r}y_{i_1} x_1^{w_{r+1}}\cdots x_r^{w_{2r}}y_{i_2} x_1^{w_{2r+1}}\cdots y_{i_k}x_1^{w_{kr+1}}\cdots x_r^{w_{kr+r}},
    \end{align*}
    where $\pi=y_{i_1}\ldots y_{i_k}$ as above. We can push the elements $y_i$ to the right using the conjugation action, and express the result in terms of the matrices $T_t$, $T_y$, $\Delta$, yielding
    \begin{align*}
        \tau_t(\overline{w}) &=\tau_t\left(\left( \Delta\begin{pmatrix} w_1 \\ \vdots \\ w_r\end{pmatrix} + T_{y_{i_1}}\Delta\begin{pmatrix} w_{r+1} \\ \vdots \\ w_{2r}\end{pmatrix} + \cdots + T_{y_{i_1}}\cdots T_{y_{i_k}}\Delta\begin{pmatrix} w_{kr+1} \\ \vdots \\ w_{kr+r}\end{pmatrix}   \right)\overline{\pi}\right) \\
        &= \left( T_t\Delta\begin{pmatrix} w_1 \\ \vdots \\ w_r\end{pmatrix} + T_tT_{y_{i_1}}\Delta\begin{pmatrix} w_{r+1} \\ \vdots \\ w_{2r}\end{pmatrix} + \cdots + T_tT_{y_{i_1}}\cdots T_{y_{i_k}}\Delta\begin{pmatrix} w_{kr+1} \\ \vdots \\ w_{kr+r}\end{pmatrix}   \right)\tau_t(\overline{\pi})\\
        &=  \left(T_t\Delta~ \vert~ T_tT_{y_{i_1}}\Delta~ \vert \cdots \vert~ T_tT_{y_{i_1}}\cdots T_{y_{i_k}}\Delta \right)\begin{pmatrix} w_1 \\ \vdots \\ w_{kr+r} \end{pmatrix}\tau_t(\overline{\pi}) \\
        &= \begin{pmatrix} A^\pi_{1,t}\cdot\phi_\pi(w) \\ \vdots \\ A^\pi_{n,t}\cdot\phi_\pi(w) \end{pmatrix} \tau_t(\overline{\pi}),
    \end{align*}
    where in the last step we have written $A^\pi_{i,t}\in\Z^{m_\pi}$ for the transpose of the $i$th row of the $n\times m_\pi$ matrix $\left(T_t\Delta~ \vert~ T_tT_{y_{i_1}}\Delta~ \vert \cdots \vert~ T_tT_{y_{i_1}}\cdots T_{y_{i_k}}\Delta \right)$, and recalled the definition of $\phi_\pi(w)$.

    Now for each $\pi\in P$ and $t\in T$ we can find (using our description of $G$) an expression $xs$, with $x\in\Z^n$ and $s\in T$, such that $\tau_t(\overline{\pi})=xs$. Writing $B^\pi_{i,t}\in\Z$ for the $i$th entry of $x$, we have 
    \begin{equation*}
        \tau_t(\overline{w}) = \left( \begin{pmatrix} A_{1,t}^\pi\cdot\phi_\pi(w) \\ \vdots \\ A_{n,t}^\pi\cdot\phi_\pi(w) \end{pmatrix} + \begin{pmatrix} B_{1,t}^\pi \\ \vdots \\ B_{n,t}^\pi \end{pmatrix} \right) s,
    \end{equation*}
    as required.

    Finally, let $B^\pi_{n+1}=\omega(\pi)$ and let $A^\pi_{n+1}\in\Z^{m_\pi}$ be the vector recording the weights of the $r$ generators in $X$, repeated $k+1$ times.

    That is, the $i$th entry of $A_{n+1}^\pi$ is equal to $r$ if $i\bmod r=0$, and equal to $i\bmod r$ otherwise.
    These clearly satisfy $\omega(w)=A^\pi_{n+1}\cdot\phi_\pi(w)+B^\pi_{n+1}$, finishing the proof.
\end{proof}

\begin{notation}\label{not:representing words using matrices}
    We use the case $t=1_G$ more frequently, and for brevity will therefore write $A_i^\pi = A_{i,1_G}^\pi$ and $B_i^\pi = B_{i,1_G}^\pi$ (for $\pi\in P$ and $1\leq i \leq n$). To be explicit, in this case we have
    \begin{equation*}
        \overline{w} = \left( \begin{pmatrix} A_{1}^\pi\cdot\phi_\pi(w) \\ \vdots \\ A_{n}^\pi\cdot\phi_\pi(w) \end{pmatrix} + \begin{pmatrix} B_{1}^\pi \\ \vdots \\ B_{n}^\pi \end{pmatrix} \right) t_\pi
        \quad\text{and}\quad
        \overline{\pi} = \begin{pmatrix} B_{1}^\pi \\ \vdots \\ B_{n}^\pi \end{pmatrix} t_\pi.
    \end{equation*}
\end{notation}
The following straightforward result uses the vectors $A_i^\pi$ and $B_i^\pi$ to provide a criterion for two words representing elements of the same coset.
\begin{proposition}[See {\cite[Proposition 4.3]{Evetts}}]\label{prop:cosetcriteria}
    Let $v\in W^\pi$ and $w\in W^\mu$ for some $\pi,\mu\in P$. Suppose that $H$ is a subgroup of $G$ contained in $\Z^n$ (so $H$ is itself free abelian of rank $d$, say) and let $\{h_1,\ldots,h_d\}\subset\Z^n$ be a basis for $H$. Then $H\overline{v}=H\overline{w}$ if and only if
    \begin{enumerate}
        \item $t_\pi=t_\mu$, and
        \item there exist integers $\lambda_1,\ldots,\lambda_d$ such that
        \begin{equation*}
            A_i^\pi\cdot\phi_\pi(v)+B_i^\pi - A_i^\mu\cdot\phi_\mu(w) - B_i^\mu = e_i\cdot\sum_{j=1}^d\lambda_jh_j
        \end{equation*}
        for each $i\in\{1,\ldots,n\}$ (where $e_i$ denotes the $i$th standard basis vector of $\Z^n$ as above).
    \end{enumerate}
\end{proposition}

\begin{definition}[Partial orders on patterned words]\label{def:order}
    Let $H$ be a subgroup of $G$ contained in $\Z^n$, with basis $\{h_1,\ldots,h_d\}$ as in Proposition \ref{prop:cosetcriteria}. For each $\pi\in P$, add standard basis vectors $A_{n+2}^\pi=e_1,\ldots,A_K^\pi=e_{m_\pi}\in\Z^{m_\pi}$ to the set $\{A_1^\pi,\ldots,A_{n+1}^\pi\}$ which ensures that the extended set $\{A_1^\pi,\ldots,A_K^\pi\}$ spans $\Z^{m_\pi}$. Define a relation $\leq_{H,\pi}$ on the words in $W^\pi$ as follows. We write $v_1\leq_{H,\pi}v_2$ if and only if either $v_1=v_2$ or there exist integers $\lambda_1,\ldots,\lambda_d$, and $i\in\{1,\ldots,K-n\}$, such that
    \begin{align*}
        A_k^\pi\cdot(\phi_\pi(v_1)-\phi_\pi(v_2)) &= e_k\cdot\sum_{j=1}^d\lambda_jh_j,  & \text{ for }k&\in\{1,\ldots,n\}, \\
        A_{k}^\pi\cdot(\phi_\pi(v_1)-\phi_\pi(v_2)) &=0,    &   \text{ for }k&\in\{n+1,\ldots,n+i-1\}, \text{and}\\
        A_{k}^\pi\cdot (\phi_\pi(v_1)-\phi_\pi(v_2)) & <0,  &   \text{ for }k&=n+i.
    \end{align*}
\end{definition}

\begin{remark}\label{rem:minreps}
    It is straightforward to verify that $\leq_{H,\pi}$ is a partial order, and moreover (via Proposition \ref{prop:cosetcriteria}) that if we restrict to words in $W^\pi$ representing a single $H$-coset, then $\leq_{H,\pi}$ is a total order and in fact a well-order (see \cite[Proposition 4.7]{Evetts} for details).

    Since $A_{n+1}^\pi\cdot\phi_\pi(v)$ is the weight of the word $v$, it follows immediately from Definition \ref{def:order} that if $v\in W^{\pi}$ is minimal (with respect to $\leq_{H,\pi}$) amongst all words in $W^\pi$ representing the coset $H\overline{v}$, then it has minimal weight amongst these words also, and is therefore a candidate to be a geodesic representative for that coset. This motivates the following definition.
\end{remark}

\begin{definition}[Minimal coset representatives]
    Let $V_H^\pi$ be the unique set of $\leq_{H,\pi}$-minimal $H$-coset representatives in $W^\pi$, that is,  \[V_H^\pi= \{v\in W^\pi : \text{if }w\in W^\pi\text{ and }\overline{v}\in H\overline{w}\text{ then }v\leq_{H,\pi} w\}.\]
\end{definition}
It is proved in \cite[Proposition 4.7]{Evetts} that each $\phi_\pi(V_H^\pi)$ is a polyhedral set. For the case where $H$ is the trivial subgroup, it is proved in \cite[Lemma 4.9]{CEL}, using other results of that article, that these polyhedral sets can be effectively constructed. This latter proof may be immediately generalised to incorporate our subgroup $H$, and here we simply state the result.
\begin{proposition}
    For each $\pi\in P$, $\phi_\pi(V_H^\pi)$ is an effectively constructible polyhedral set.
\end{proposition}

We now prove the main result of this section.
\begin{proof}[Proof of Theorem \ref{thm:cosetreps}]
    The sets of words $V_H^\pi$ that we have constructed above satisfy some of the conditions required of the claimed sets $U_H^\pi$ of Theorem \ref{thm:cosetreps}. Namely, each $\phi_\pi(V_H^\pi)$ is an effectively constructible polyhedral set, and the union $\bigcup_{\pi\in P}V_H^\pi$ contains a representative for every coset in $G/H$. Furthermore, by Remark \ref{rem:minreps}, this union will contain a geodesic representative for every coset. However, there may be a coset that is represented by a word in $V_H^\pi$ and a word in $V_H^\mu$, for different patterns $\pi$ and $\mu$. To prove the theorem, we reduce each $V_H^\pi$ to a subset $U_H^\pi$, removing each such multiple representation, but retaining the property that each $\phi_\pi(U_H^\pi)$ is polyhedral. Our argument follows the proof of \cite[Theorem 4.2]{Evetts}, but emphasises effectiveness.

    For each $\pi\in P$, denote by $\mathcal{A}^\pi\colon \Z^{m_\pi}\to\Z^{n+1}$ the integer affine map given by
    \begin{equation*}
        \mathcal{A}^\pi\colon\phi_\pi(w)\mapsto \begin{pmatrix} A^\pi_1\cdot\phi_\pi(w) + B^\pi_1 \\ \vdots \\ A^\pi_{n+1}\cdot \phi_\pi(w) + B^\pi_{n+1}\end{pmatrix}.
    \end{equation*}
    As in Proposition \ref{prop:cosetcriteria}, let $\{h_1,\ldots,h_d\}$ be a basis for $H$. Let $\bm{1}_j$ denote the vector (of appropriate dimension) with a 1 in the $j$th position and zeroes elsewhere, and let \[E_i=\begin{pmatrix} 0 \\ \vdots \\ 0 \\ e_i\cdot h_1 \\ \vdots \\ e_i\cdot h_d \end{pmatrix}\in\Z^{2n+2+d}.\] Letting $p_{m}$ denote projection onto the first $m$ coordinates (which is an affine integer map), define polyhedral subsets of $\Z^{2n+2}$ as follows.
    \begin{align*}
    \Theta_0 &= p_{2n+2}\left( \bigcap_{i=1}^n\left\{\theta\in\Z^{2n+2+d}\colon \theta\cdot\left(\bm{1}_i-\bm{1}_{i+n+1}-E_i\right)=0\right\}\right)\\
        \Theta_= &= \Theta_0 \cap \{\theta\in\Z^{2n+2} : \theta\cdot(\bm{1}_{n+1}-\bm{1}_{2n+2})=0\}\\
        \Theta_> &= \Theta_0 \cap \{\theta\in\Z^{2n+2} : \theta\cdot(\bm{1}_{n+1}-\bm{1}_{2n+2})>0\}.
    \end{align*}
    Here $\Theta_0$ is used to detect pairs of words representing the same $H$-coset. That is, $(\mathcal{A}^\pi\circ\phi_\pi(v),\mathcal{A}^\mu\circ\phi_\mu(w))\in\Theta_0$ precisely when there is a tuple of coefficients $(\lambda_1,\ldots,\lambda_d)$ representing an element $h$, in terms of the basis $\{h_1,\ldots,h_d\}$, such that $\overline{v}=h\overline{w}$ (via the criterion of Proposition \ref{prop:cosetcriteria}). Building on that, $\Theta_=$ will be used to detect pairs of words of equal weight representing the same $H$-coset, whereas $\Theta_>$ will detect pairs of words where the first has larger weight, but they still represent the same $H$-coset.

   For each pair $(\pi,\mu)\in P^2$ satisfying $t_\pi=t_\mu$, define two subsets of $V_H^\pi$,
    \begin{equation}\label{eq:Requals}
        R_=^{\pi,\mu} = (\mathcal{A}^\pi\circ\phi_\pi)^{-1}\circ p_{n+1}\bigg( \big((\mathcal{A}^\pi\circ\phi_\pi(V_H^\pi)) \times (\mathcal{A}^\mu\circ\phi_\mu(V_H^\mu))\big)\: \cap\: \Theta_=\bigg)
    \end{equation}
    and
    \begin{equation}\label{eq:Rgreater}
        R_>^{\pi,\mu} = (\mathcal{A}^\pi\circ\phi_\pi)^{-1}\circ p_{n+1}\bigg( \big((\mathcal{A}^\pi\circ\phi_\pi(V_H^\pi)) \times (\mathcal{A}^\mu\circ\phi_\mu(V_H^\mu))\big)\: \cap\: \Theta_>\bigg).
    \end{equation}
    It follows from Propositions \ref{prop:representing words using matrices} and \ref{prop:cosetcriteria} that two words $w\in V_H^\pi$ and $v\in V_H^\mu$ represent the same $H$-coset and satisfy $\omega(w)>\omega(v)$ if and only if $t_\pi=t_\mu$ and $\left(\mathcal{A}^\pi(\phi_\pi(w)),\mathcal{A}^\mu(\phi_\mu(v))\right)\in\Theta_>$. Similarly, $w$ and $v$ represent the same $H$ coset and satisfy $\omega(w)=\omega(v)$ if and only if $t_\pi=t_\mu$ and $\left(\mathcal{A}^\pi(\phi_\pi(w)),\mathcal{A}^\mu(\phi_\mu(v))\right)\in\Theta_=$. 
    
    From this observation, we have that
    \begin{align*}
         R^{\pi,\mu}_= = \{w\in V_H^\pi : \exists v\in V_H^\mu~\text{such that}~\overline{w}\in H\overline{v}~\text{and}~\omega(w)=\omega(v)\},
    \end{align*}
    and
    \begin{align*}
         R^{\pi,\mu}_> = \{w\in V_H^\pi : \exists v\in V_H^\mu~\text{such that}~\overline{w}\in H\overline{v}~\text{and}~ \omega(w)>\omega(v)\}.
    \end{align*}
    Removing from $V_H^\pi$ all of the sets $R_>^{\pi,\mu}$ ensures that there are no non-geodesic coset representatives. To ensure that we do not have multiple geodesic representatives of the same coset, we use the sets $R_=^{\pi,\mu}$, but we must make a choice as to which of two equally weighted words to remove. To that end, fix an arbitrary total order $<$ on the set $P$, and remove the set $R_=^{\pi,\mu}$ precisely when $\mu<\pi$. In other words, if two coset representatives have equal weight, choose to keep the one whose pattern comes first in the order. More formally, for each $\pi$, let
    \begin{equation*}
        U_H^\pi = V_H^\pi \setminus \left( \bigcup_{\substack{\pi\neq\mu\\ t_\pi=t_\mu}} R_>^{\pi,\mu} \cup \bigcup_{\substack{\mu<\pi\\t_\pi=t_\mu}} R_=^{\pi,\mu} \right).
    \end{equation*}
    By construction, the disjoint union $\bigcup_{\pi\in P}U_H^\pi$ consists of exactly one geodesic representative for each $H$-coset. Furthermore, since $\Theta_=$, $\Theta_>$, and all $\phi_\pi(V_H^\pi)$ are polyhedral and effectively constructible, so too are the sets
    \begin{equation*}
        \phi_\pi(U_H^\pi) = \phi_\pi(V_H^\pi)\setminus \left( \bigcup_{\substack{\pi\neq\mu\\ t_\pi=t_\mu}} \phi_\pi(R_>^{\pi,\mu}) \cup \bigcup_{\substack{\mu<\pi\\t_\pi=t_\mu}} \phi_\pi(R_=^{\pi,\mu}) \right),
    \end{equation*}
    by Proposition \ref{prop:polyhedral sets basic properties}.
\end{proof}
\begin{remark}
    While it would represent a detour from the current article, we are confident that the related proof of \cite[Theorem 4.1]{Evetts} can also be made effective, with the result that the rational `coset growth series' for any subgroup of $G$ can be explicitly computed. 
\end{remark}

To finish the section, we record a result which allows us to select a set of minimal weight representatives of a polyhedral set of tuples of patterned words, while preserving rational growth. This result will be used specifically to pick minimal representatives for every twisted conjugacy class to conclude the proof of Theorem \ref{thmLetter:mainthm}. To simplify notation we first make the following definition.
\begin{definition}\label{def:productphi}
    Let $\bm{\pi}=(\pi_1,\ldots,\pi_d)\in P^d$ be some $d$-tuple of patterns, and write $m_{\bm{\pi}}=m_{\pi_1}+\cdots+m_{\pi_d}$. Then define the natural product map
    \[\phi_{\bm{\pi}}\colon W^{\pi_1}\times\cdots\times W^{\pi_d}\to\N^{m_{\bm{\pi}}}: (w^{(1)},\ldots,w^{(d)})\mapsto (\phi_{\pi_1}(w^{(1)}),\ldots,\phi_{\pi_d}(w^{(d)})).\]
\end{definition}
\begin{proposition}\label{prop:constructing language}
    Let $V\subset W^{\pi_1}\times\cdots\times W^{\pi_d}$ be a set of $d$-tuples of patterned words, for some $d$-tuple of patterns $\bm{\pi}=(\pi_1,\ldots,\pi_d)\in P^d$, such that $\phi_{\bm{\pi}}(V)$ is a polyhedral set.
    Then there exists a set of words $L_V$ over the generating set $S$ such that
    \begin{enumerate}
        \item for each $(v_1,\ldots,v_d)\in V$, the set $L_V$ contains exactly one of the words $v_i$, satisfying $\omega(v_i)\leq\omega(v_j)$ for $1\leq j\leq d$, and
        \item $L_V$ has explicitly computable $\N$-rational growth series.
    \end{enumerate}
\end{proposition}
\begin{proof}
    The existence of a set $L_V$ with the described properties is precisely the statement of \cite[Lemma 2.21]{Evetts}. In fact, $L_V$ is effectively constructed in that article as a finite disjoint union of sets $L_V^{\pi_i}\subset S^*$, where each $\phi_{\pi_i}(L_V^{\pi_i})$ is an effectively constructible polyhedral set. By Proposition \ref{prop:Nrational}, the growth series of each $L_V$ is therefore $\N$-rational and explicitly computable.
\end{proof}

\section{Twisted conjugacy growth of virtually abelian groups}\label{sec:main}
This section describes the twisted conjugacy classes of an arbitrary finitely generated virtually abelian group $G$. We argue that the twisted conjugacy growth series is $\N$-rational. Moreover, if the fully characteristic finite index normal subgroup $\Z^n$ of $G$ and $\varphi\in \End(G)$ are given as in Definition \ref{def:input data}, the exact form of the rational series may be computed. Let $S$ denote the extended generating set (as introduced in Definition \ref{def:extendedgenset}) with respect to a finite weighted generating set $\Sigma$ of $G$ and let $T$ be a transversal for the $\Z^n$-cosets in $G$, with the group identity representing the identity coset. Hence, we can write every $g\in G$ uniquely as
\[ g=x_gt_g \quad \text{with } x_g\in \Z^n \text{ and } t_g\in T. \]
We will use, for a word $w\in S^\ast$, the notation $t_w$ and $x_w$ instead of $t_{\overline{w}}$ and $x_{\overline{w}}$, respectively.

\begin{definition}
    For every $g\in G$ and $\varphi\in \End(G)$ we define the subgroup $H(g)$ of $\Z^n$ by setting
    \[ H(g)=\Image(\Id_{\Z^n}-\tau_g\circ \varphi\vert_{\Z^n}). \]
\end{definition}

Since $\Z^n\lhd G$ is fully characteristic and normal, for $g\in G$, the map $\tau_g\circ \varphi\vert_{\Z^n}$ is indeed an endomorphism of $\Z^n$ and thus $H(g)$ is well-defined. Moreover, note that if $g\in \Z^n h$ for some $g,h\in G$, then $H(g)=H(h)$. In other words, $H(g)$ only depends on the $\Z^n$-coset of $g$. In particular, if $w\in W^\pi$ for a pattern $\pi\in P$, then $H(\overline{w})=H(\overline{\pi})$.

Recall that, as in the proof of Proposition \ref{prop:representing words using matrices}, $\tau_{\overline{\pi}}$ can be explicitly expressed as a matrix. Moreover, we can find a matrix representing the linear map $\Id_{\Z^n}-\tau_g\circ\varphi\vert_{\Z^n}$, and hence a $\Z$-basis for its image, $H(g)$. 

\begin{notation}
    For two subsets $A$ and $B$ of $G$ and an endomorphism $\varphi\in \End(G)$ we write ${}^{B,\varphi} A$ for the set of $\varphi$-twisted conjugates of elements from $A$ by elements from $B$, i.e.
\[{}^{B,\varphi} A=\{ba\varphi(b)^{-1} : a\in A,\: b\in B\}.\]
\end{notation}

\begin{lemma}\label{lem:union tc class}
	If $g\in G$ and $\varphi\in \End(G)$, then 
	\[ [g]_\varphi = \bigcup_{t\in T} {}^{\Z^nt,\varphi}\{g\}. \]
        Moreover, for every $t\in T$ we get
        \[{}^{\Z^nt,\varphi}\{g\}=H(tg\varphi(t)^{-1})tg\varphi(t)^{-1}.\]
\end{lemma}
\begin{proof}
	Fix $g\in G$ and note that
        \[[g]_\varphi = \{xt\: g\: \varphi(xt)^{-1} : x\in \Z^n,\: t\in T\} = \bigcup_{t\in T} {}^{\Z^nt,\varphi}\{g\}.\]
        Moreover, if we fix some $t\in T$ it follows that
        \begin{align*}
		{}^{\Z^nt,\varphi}\{g\} &= \{xt\: g\: \varphi(xt)^{-1} : x\in \Z^n\} = \{x \tau_{tg\varphi(t)^{-1}}(\varphi(x)^{-1}) tg\varphi(t)^{-1} : x\in \Z^n\}\\
					            &= \{x - (\tau_{tg\varphi(t)^{-1}}\circ\varphi)(x) : x\in \Z^n\} tg\varphi(t)^{-1} = H(tg\varphi(t)^{-1})tg\varphi(t)^{-1}
	\end{align*}
        where we have used once more additive notation within the abelian group $\Z^n$ for convenience.
\end{proof}

Hence, every twisted conjugacy class of $G$ is a finite union of cosets of subgroups of $\Z^n$. The next lemma describes when another element belongs to one of these cosets appearing in this union.

\begin{lemma}\label{lem:description tconj by tZ^n}
    Let $g\in G$, $\varphi\in \End(G)$ and $t\in T$. Then, for every $h\in G$, the following holds
    \[h\in {}^{\Z^nt,\varphi}\{g\} \quad \Longleftrightarrow \quad \begin{cases}
                                                                        t_h\in tt_g\varphi(t)^{-1}\Z^n\text{, and}\\
                                                                        x_h\in H(tt_g\varphi(t)^{-1})+\tau_t(x_g)+x_{tt_g\varphi(t)^{-1}}
                                                                    \end{cases}.\]
\end{lemma}
\begin{proof}
    From Lemma \ref{lem:union tc class}, we observe, for $h\in G$, that
    \begin{align*}
        h\in {}^{\Z^nt,\varphi}\{g\} \quad &\Longleftrightarrow \quad h\in H(tg\varphi(t)^{-1})tg\varphi(t)^{-1}\\
                                            &\Longleftrightarrow \quad \exists z\in \Z^n: h=z \tau_{tg\varphi(t)^{-1}}(\varphi(z)^{-1}) tg\varphi(t)^{-1}\\
                                            &\Longleftrightarrow \quad \exists z\in \Z^n: x_ht_h=z \tau_{tt_g\varphi(t)^{-1}}(\varphi(z)^{-1})\tau_t(x_g)tt_g\varphi(t)^{-1}\\
                                            &\Longleftrightarrow \quad \begin{cases}
                                                                            t_h\in tt_g\varphi(t)^{-1}\Z^n\text{, and}\\
                                                                            x_h\in H(tt_g\varphi(t)^{-1})+\tau_t(x_g)+tt_g\varphi(t)^{-1}t_h^{-1}
                                                                        \end{cases}.
    \end{align*}
    If $t_h\in tt_g\varphi(t)^{-1}\Z^n$, then since $T$ is a transversal it holds by construction that $t_{tt_g\varphi(t)^{-1}}=t_h$ and thus $tt_g\varphi(t)^{-1}=x_{tt_g\varphi(t)^{-1}}t_h$.
\end{proof}

\begin{lemma}\label{lem:Z^n conjugates}
    Let $g\in G$ and $h\in {}^{\Z^n,\varphi} \{g\}$, then for every $t\in T$ it holds that ${}^{\Z^n t,\varphi} \{g\}= {}^{\Z^n t,\varphi} \{h\}$.
\end{lemma}
\begin{proof}
    Since $h\in H(g)g$, there exists some $z_0\in \Z^n$ such that $h=z_0g\varphi(z_0)^{-1}g^{-1}g=z_0g\varphi(z_0)^{-1}$. Using that $\tau_t(z_0)\in \Z^n$ is fixed, we obtain, for every $t\in T$, that
    \begin{align*}
        {}^{\Z^n t,\varphi} \{h\} &= \{xth\varphi(xt)^{-1} : x\in \Z^n\} = \{xtz_0g\varphi(xtz_0)^{-1} : x\in \Z^n\} \\
                                &= \{x \tau_t(z_0) tg\varphi(x \tau_t(z_0)t)^{-1} : x\in \Z^n\} = \{\tilde{x}tg\varphi(\tilde{x}t)^{-1} : \tilde{x}\in \Z^n\}\\
                                &= {}^{\Z^n t,\varphi} \{g\}.
    \end{align*}
\end{proof}

Recall that, by Theorem \ref{thm:cosetreps}, we have for each pattern $\pi\in P$ a set of geodesic representatives $U^\pi_{H(\overline{\pi})}\subset S^\ast$ for $H(\overline{\pi})$-cosets. We use these sets to describe unique minimal representatives for each $H(\cdot)$-coset appearing in the description of a fixed twisted conjugacy class, as in Lemma \ref{lem:union tc class}. Write $T=\{t_1,\ldots,t_D\}$, with $t_1=1_G$, and $D=[G:\Z^n]$.
\begin{definition}\label{def:Cpi}
    We call a $D$-tuple $\bm{\pi}=(\pi_1,\dots,\pi_D)$ of patterns, i.e. in $P$, \emph{permissible}  if $t_j\overline{\pi_1}\varphi(t_j)^{-1}\in \Z^n\overline{\pi_j}$ for $2\leq j \leq D$.\\
    For every permissible $D$-tuple of patterns $\bm{\pi}=(\pi_1,\dots,\pi_D)$, define
    \[C(\bm{\pi})=\left\{\left(w^{(1)},\dots,w^{(D)}\right) \in U_{H(\overline{\pi_1})}^{\pi_1} \times \dots \times U_{H(\overline{\pi_D})}^{\pi_D} : \begin{array}{l}\overline{w^{(j)}}\in {}^{\Z^nt_j,\varphi}\left\{\overline{w^{(1)}}\right\}\\ \text{for }2\leq j\leq D\end{array}\right\}.\]
    In other words, $C(\bm{\pi})$ contains exactly the $D$-tuples of words where each $w^{(j)}$ is the unique minimal representative (with pattern $\pi_j$) of its $H(\overline{\pi_j})$-coset and each $\overline{w^{(j)}}$ is $\varphi$-twisted conjugate to $\overline{w^{(1)}}$ by an element of $\Z^nt_j$.
\end{definition}

\begin{lemma}\label{lem:elements C(pi)}
    Each element of $C(\bm{\pi})$ consists of $D$-tuples of words representing elements of the same $\varphi$-twisted conjugacy class and every $\varphi$-twisted conjugacy class is represented by an element of $C(\bm{\pi})$, for some permissible $D$-tuple $\bm{\pi}\in P^D$.\\
    Moreover, the weight of the $\varphi$-twisted conjugacy class is realised by the word(s) of smallest length in this $D$-tuple, i.e. for every $\left(w^{(1)},\dots,w^{(D)}\right)\in C(\bm{\pi})$ it holds that
    \[\omega\left(\left[\overline{w^{(1)}}\right]_\varphi\right)=\min \{\omega(\overline{w^{(1)}}),\dots, \omega(\overline{w^{(D)}})\}.\]
\end{lemma}
\begin{proof}
    Fix a $D$-tuple $\left(w^{(1)},\dots,w^{(D)}\right)\in C(\bm{\pi})$. By definition, $\overline{w^{(j)}}\sim_\varphi \overline{w^{(1)}}$ for $2\leq j\leq D$ and thus these elements represent the same $\varphi$-twisted conjugacy class $\left[\overline{w^{(1)}}\right]_\varphi$.\\
    Moreover, observe that Lemma \ref{lem:union tc class} implies
    \[\left[\overline{w^{(1)}}\right]_\varphi = \bigcup_{j=1}^D H(t_j\overline{w^{(1)}}\varphi(t_j)^{-1})t_j\overline{w^{(1)}}\varphi(t_j)^{-1} = \bigcup_{j=1}^D H(\overline{\pi_j})t_j\overline{w^{(1)}}\varphi(t_j)^{-1}\]
    where we used that $t_j\overline{w^{(1)}}\varphi(t_j)^{-1}\in t_j\overline{\pi_1}\varphi(t_j)^{-1}\Z^n=\overline{\pi_j}\Z^n$. Observe that, for $2\leq j\leq D$,
    \[\overline{w^{(j)}}\in {}^{\Z^nt_j,\varphi}\left\{\overline{w^{(1)}}\right\}=H(t_j\overline{w^{(1)}}\varphi(t_j)^{-1})t_j\overline{w^{(1)}}\varphi(t_j)^{-1}=H(\overline{\pi_j})t_j\overline{w^{(1)}}\varphi(t_j)^{-1}.\]
    Since $w^{(j)}\in U^{\pi_j}_{H(\overline{\pi_j})}$, the evaluation $\overline{w^{(j)}}$ is the unique minimal representative of its $H(\overline{\pi_j})$-coset $H(\overline{\pi_j})t_j\overline{w^{(1)}}\varphi(t_j)^{-1}$. Hence, the word(s) in $\left\{w^{(1)},\dots,w^{(D)}\right\}$ of smallest length realise(s) the weight of $\left[\overline{w^{(1)}}\right]_\varphi$.

    To see that every $\varphi$-twisted conjugacy class is represented by an element of $C(\bm{\pi})$ for some $\bm{\pi}$, consider such a class $[g]_\varphi = \bigcup_{t\in T} H(tg\varphi(t)^{-1})tg\varphi(t)^{-1}$. By Theorem \ref{thm:cosetreps}, for each $t\in T$ there exists some pattern $\pi_t\in P$ such that the coset $H(tg\varphi(t)^{-1})tg\varphi(t)^{-1}$ is uniquely represented by a word in $U^{\pi_t}_{H(tg\varphi(t)^{-1})}$. Thus, there exists a unique $D$-tuple of words that all represent elements of $[g]_\varphi$, with patterns $\pi_{t_1},\ldots,\pi_{t_D}\in P$, of the form
    
    \[(w^{(1)},\ldots,w^{(D)})\in U^{\pi_{t_1}}_{H(t_1g\varphi(t_1)^{-1})}\times \cdots \times U^{\pi_{t_D}}_{H(t_Dg\varphi(t_D)^{-1})}.\]
    
    We argue that $(\pi_{t_1},\dots,\pi_{t_D})$ is permissible and $(w^{(1)},\ldots,w^{(D)})\in C(\bm{\pi})$. We first remark that $H(t_jg\varphi(t_j)^{-1})=H(\overline{\pi_{t_j}})$ for $1\leq j\leq D$. Indeed, note that since $\overline{w^{(j)}}\in H(t_jg\varphi(t_j)^{-1})t_jg\varphi(t_j)^{-1}$, the elements $t_jg\varphi(t_j)^{-1}$ and $\overline{w^{(j)}}$ represent the same $H(t_jg\varphi(t_j)^{-1})$-coset (and hence the same $\Z^n$-coset). Furthermore, since $w^{(j)}$ has pattern $\pi_{t_j}$, $\overline{w^{(j)}}$ and $\overline{\pi_{t_j}}$ also represent the same $\Z^n$-coset. Thus, by definition of the subgroups $H(\cdot)$, we get that $H(t_jg\varphi(t_j)^{-1})=H(\overline{\pi_{t_j}})$.
    
    Recall that $t_1=1_G$ by definition of $T$. Since, for $1\leq j\leq D$, it holds that $t_jg\varphi(t_j)^{-1}$ and $\overline{\pi_{t_j}}$ represent the same $\Z^n$-coset, we obtain that $\Z^ng=\Z^n\overline{\pi_{t_1}}$ and thus for $2\leq j \leq D$ that
    \[\Z^n\overline{\pi_{t_j}}=\Z^n t_jg\varphi(t_j)^{-1}= \Z^n t_j\overline{\pi_{t_1}}\varphi(t_j)^{-1}.\]
    Hence, it remains to argue that $\overline{w^{(j)}}\in {}^{\Z^nt_j,\varphi}\left\{\overline{w^{(1)}}\right\}$ for $ 2\leq j\leq D$. Since $\overline{w^{(1)}}\in H(t_1g\varphi(t_1)^{-1})t_1g\varphi(t_1)^{-1}=H(g)g$, Lemmas \ref{lem:union tc class} and \ref{lem:Z^n conjugates} imply, for $2\leq j\leq D$, that
    \[\overline{w^{(j)}}\in H(t_jg\varphi(t_j)^{-1})t_jg\varphi(t_j)^{-1}={}^{\Z^nt_j,\varphi}\left\{g\right\}={}^{\Z^nt_j,\varphi}\left\{\overline{w^{(1)}}\right\}.\]
    \end{proof}

\begin{lemma}\label{lem:C(pi) polyhedral set}
    The image $\phi_{\bm{\pi}}(C(\bm{\pi}))\subset \N^{m_{\bm{\pi}}}$ of the map $\phi_{\bm{\pi}}$ in Definition \ref{def:productphi} is an effectively constructible polyhedral set for every permissible $D$-tuple $\bm{\pi}\in P^d$.
\end{lemma}
\begin{proof}
    We want to describe the condition that $\overline{w^{(j)}}\in {}^{\Z^nt_j,\varphi}\left\{\overline{w^{(1)}}\right\}$ (for $2\leq j \leq D$) using matrices. This will allow us to use the properties of polyhedral sets from Proposition \ref{prop:polyhedral sets basic properties} combined with Theorem \ref{thm:cosetreps} to conclude the result.
    
    For $2\leq j\leq D$ we recall that $t_j\overline{\pi_1}\varphi(t_j)^{-1}\in \Z^n\overline{\pi_j}$ and thus
    \[t_{w^{(j)}} = t_{\pi_j} \in t_j\overline{\pi_1}\varphi(t_j)^{-1}\Z^n = t_jt_{w^{(1)}}\varphi(t_j)^{-1}\Z^n.\]
    Hence, by Lemma \ref{lem:description tconj by tZ^n} it follows that
    \begin{align*}
        \overline{w^{(j)}}\in {}^{\Z^nt_j,\varphi}\left\{\overline{w^{(1)}}\right\}\quad
        &\Longleftrightarrow \quad x_{w^{(j)}}\in H(t_jt_{w^{(1)}}\varphi(t_j)^{-1})+\tau_{t_j}(x_{w^{(1)}})+x_{t_jt_{w^{(1)}}\varphi(t_j)^{-1}}\\
        &\Longleftrightarrow \quad x_{w^{(j)}}\in H(\overline{\pi_j})+\tau_{t_j}(x_{w^{(1)}})+x_{t_jt_{\pi_1}\varphi(t_j)^{-1}}.                   
    \end{align*}
    Recall that Proposition \ref{prop:representing words using matrices} gives
        \[x_{w^{(j)}}=
    \begin{pmatrix}
        A_1^{\pi_j}\cdot \phi_{\pi_j}({w^{(j)}})\\
        \vdots\\
        A_n^{\pi_j}\cdot \phi_{\pi_j}({w^{(j)}})
    \end{pmatrix}+
    \begin{pmatrix}
        B_1^{\pi_j}\\
        \vdots\\
        B_n^{\pi_j}
    \end{pmatrix}\]
    and that its proof yields
    \begin{align*}
        \tau_{t_j}(x_{w^{(1)}})&=\tau_{t_j}\left(\overline{w^{(1)}}\right)\tau_{t_j}\left(t_{\pi_1}^{-1}\right) = 
        \begin{pmatrix}
            A_{1,t_j}^{\pi_1}\cdot \phi_{\pi_1}(w^{(1)})\\
            \vdots\\
            A_{n,t_j}^{\pi_1}\cdot \phi_{\pi_1}(w^{(1)})
        \end{pmatrix}
        \tau_{t_j}\left(\overline{\pi_1}\right) \tau_{t_j}\left(t_{\pi_1}^{-1}\right)\\
        &=\begin{pmatrix}
            A_{1,t_j}^{\pi_1}\cdot \phi_{\pi_1}(w^{(1)})\\
            \vdots\\
            A_{n,t_j}^{\pi_1}\cdot \phi_{\pi_1}(w^{(1)})
        \end{pmatrix}
        +\tau_{t_j}\begin{pmatrix}
            B_1^{\pi_1}\\
            \vdots\\
            B_n^{\pi_1}
        \end{pmatrix}
    \end{align*}
    where in the last step we used Notation \ref{not:representing words using matrices}. Define $C^j=(c_{il}^j)\in \Z^{n\times n}$ and $D_{i,t_j}\in \Z$ (for $1\leq i\leq n$) by setting
    \[C^j=\Id_{\Z^n}-\tau_{t_{\pi_j}}\circ \varphi\vert_{\Z^n} \quad \text{and} \quad 
    \begin{pmatrix}
        D_{1,t_j}\\
        \vdots\\
        D_{n,t_j}
    \end{pmatrix}=\tau_{t_j}\begin{pmatrix}
            B_1^{\pi_1}\\
            \vdots\\
            B_n^{\pi_1}
        \end{pmatrix}+x_{t_jt_{\pi_1}\varphi(t_j)^{-1}}.\]
    Hence, we find that 
    \begin{align*}
        &\overline{w^{(j)}}\in {}^{\Z^nt_j,\varphi}\left\{\overline{w^{(1)}}\right\}\quad
        \Longleftrightarrow \quad x_{w^{(j)}}\in H(\overline{\pi_j})+\tau_{t_j}(x_{w^{(1)}})+x_{t_jt_{\pi_1}\varphi(t_j)^{-1}}\\
        &\Longleftrightarrow \quad \exists z\in \Z^n:\\
	&\hspace{10pt}\begin{pmatrix}
        A_1^{\pi_j}\cdot \phi_{\pi_j}({w^{(j)}})\\
        \vdots\\
        A_n^{\pi_j}\cdot \phi_{\pi_j}({w^{(j)}})
    \end{pmatrix}+
    \begin{pmatrix}
        B_1^{\pi_j}\\
        \vdots\\
        B_n^{\pi_j}
    \end{pmatrix} = 
    \begin{pmatrix}
        A_{1,t_j}^{\pi_1}\cdot \phi_{\pi_1}(w^{(1)})\\
        \vdots\\
        A_{n,t_j}^{\pi_1}\cdot \phi_{\pi_1}(w^{(1)})
    \end{pmatrix}
    +\begin{pmatrix}
        D_{1,t_j}\\
        \vdots\\
        D_{n,t_j}
    \end{pmatrix}
    +C^j(z)\\
    &\Longleftrightarrow \quad \exists (z_1,\dots,z_n)\in \Z^n:\: \forall 1\leq i\leq n:\\
    &\hspace{45pt}A_{i,t_j}^{\pi_1}\cdot \phi_{\pi_1}(w^{(1)})-A_i^{\pi_j}\cdot \phi_{\pi_j}({w^{(j)}})+\sum_{l=1}^n c_{il}^jz_l = B_i^{\pi_j}-D_{i,t_j}.
    \end{align*}
    Now define for $1\leq i \leq n$ and $2\leq j \leq D$ vectors $M_i^j(\bm{\pi})$ and $N_i^j(\bm{\pi})$. The vector $M_i^j(\bm{\pi})$ comprises $D$ entries, the $l$-th entry being an element of $\Z^{m_{\pi_l}}$, whereas the vector $N_i^j(\bm{\pi})$ simply lies in $\Z^{(D-1)n}$. They are defined as follows
    \[\hspace{-55pt}M_i^j(\bm{\pi})=
    \begin{pNiceMatrix}[last-col=2]
	A_{i,t_j}^{\pi_1} & \\
        0 & \\
        \vdots & \\
        0 & \\
        -A_i^{\pi_j} & \hspace{5pt}j\text{-th entry}\\
        0 & \\
        \vdots & \\
        0 &
    \end{pNiceMatrix}
    \quad \text{and} \quad N_i^j(\bm{\pi})=
    \begin{pNiceMatrix}
        0\\\vdots\\0\\
        c_{i1}^j\\\vdots\\c_{in}^j\\
        0\\\vdots\\0
        \CodeAfter
        \SubMatrix.{1-1}{3-1}\}[right-xshift=0.5em,name=A]
        \SubMatrix.{4-1}{6-1}\}[right-xshift=0.5em,name=B]
        \SubMatrix.{7-1}{9-1}\}[right-xshift=0.5em,name=C]
        \tikz \node [right] at (A-right.east) {$(j-2)n$ zeroes} ;
        \tikz \node [right] at (B-right.east) {$i$-th row of $C^j$} ;
        \tikz \node [right] at (C-right.east) {$(d-j)n$ zeroes} ;
    \end{pNiceMatrix}.\]
    Thus we obtain that
    \begin{align*}
        \overline{w^{(j)}}\in {}^{\Z^nt_j,\varphi}\left\{\overline{w^{(1)}}\right\}\quad \Longleftrightarrow \quad 
        &\exists z\in \Z^n:\: \forall 1\leq i\leq n:\\
        &M_i^j(\bm{\pi})\cdot
        \begin{pmatrix}
            \phi_{\pi_1}(w^{(1)})\\
            \vdots\\
            \phi_{\pi_D}({w^{(D)}})
        \end{pmatrix}
        +N_i^j(\bm{\pi})\cdot z = B_i^{\pi_j}-D_{i,t_j}.
    \end{align*}
    Note that for $1\leq i \leq n$ and $2\leq j \leq D$, the sets
    \[E_i^j=\left\{k\in \N^{m_{\bm{\pi}}+(D-1)n} :
        \begin{pmatrix}
            M_i^j(\bm{\pi})\\
            N_i^j(\bm{\pi})
        \end{pmatrix}\cdot k =  B_i^{\pi_j}-D_{i,t_j}\right\}\]
    are elementary sets. Moreover, if we denote with $p_{m_{\bm{\pi}}}:\N^{m_{\bm{\pi}}+(D-1)n}\to \N^{m_{\bm{\pi}}}$ the projection onto the first $m_{\bm{\pi}}$ coordinates, then
    \[\phi_{\bm{\pi}}(C(\bm{\pi}))=(\phi_{\pi_1}(U_{H(\overline{\pi_1})}^{\pi_1}) \times \dots \times \phi_{\pi_D}(U_{H(\overline{\pi_D})}^{\pi_D})) \times \bigcap_{j=2}^D \bigcap_{i=1}^n p_{m_{\bm{\pi}}}(E_i^j)\]
    is an effectively constructible polyhedral set by Proposition \ref{prop:polyhedral sets basic properties} and Theorem \ref{thm:cosetreps}.
\end{proof}

We are now in a position to prove our main theorem.
\begin{theoremLetter}\label{thm:mainthm}
    Let $G$ be a finitely generated virtually abelian group with a finite weighted (monoid) generating set $\Sigma$ and $\varphi\in \End(G)$. Then the $\varphi$-twisted conjugacy growth series of $G$ with respect to $\Sigma$ is an $\N$-rational function. Moreover, given a description for $G$ as in Definition \ref{def:input data}, the $\N$-rational function is explicitly computable.
\end{theoremLetter}

\begin{proof}
    Recall the set $P$ of patterns from Definition \ref{def:patterns}. By Lemma \ref{lem:elements C(pi)}, for every $\varphi$-twisted conjugacy class of $G$ we can pick a permissible $D$-tuple $\bm{\pi}\in P^D$ such that $C(\bm{\pi})$ contains a $D$-tuple of words representing elements of this class. Moreover, the weight of the $\varphi$-twisted conjugacy class is realised by the word(s) of smallest length in this $D$-tuple. We now reduce the set of possible $\bm{\pi}$s to avoid overlap of candidates for the same $\varphi$-twisted conjugacy class.
    
    Consider the set $R$ of permissible $D$-tuples of patterns
    \[R=\{(\pi_1,\dots,\pi_D)\in P^D : (\pi_1,\dots,\pi_D) \text{ is permissible}\}.\]
    We reduce $R$ to a subset $R_0$ such that no two elements of $R_0$ are permutations of each other. Note that since we have computed the finite set $P$, we can also compute $R_0$. In particular, for every $\varphi$-twisted conjugacy class we can still pick $\bm{\pi}\in R_0$ and a $D$-tuple in $C(\bm{\pi})$ representing this class.

    Moreover, no $\varphi$-twisted conjugacy class is represented by more than one element of $R_0$, i.e. there do not exist different $\bm{\pi},\bm{\pi'}\in R_0$ and $D$-tuples of words $\left(w^{(1)},\dots,w^{(D)}\right)\in C(\bm{\pi})$ and $\left((w')^{(1)},\dots,(w')^{(D)}\right)\in C(\bm{\pi'})$ representing the same $\varphi$-twisted conjugacy class. Indeed, since we removed all the permutations of elements in $R$, every candidate $D$-tuple determines a unique element of $R_0$. Also by Theorem \ref{thm:cosetreps} the candidates are unique.

    Thus for every $\varphi$-twisted conjugacy class $[g]_\varphi$ (for $g\in G$) there is a unique $\bm{\pi}\in R_0$ and a unique $D$-tuple $\left(w^{(1)},\dots,w^{(D)}\right)\in C(\bm{\pi})$ of words representing elements of this class. Moreover, by Lemma \ref{lem:elements C(pi)},
    \[\omega\left([g]_\varphi\right)=\min \{\omega(\overline{w^{(1)}}),\dots, \omega(\overline{w^{(D)}})\}.\]
    Since $\phi_{\bm{\pi}}(C(\bm{\pi}))\subset \N^{m_{\bm{\pi}}}$ is an effectively constructible polyhedral set, Proposition \ref{prop:constructing language} yields, for $\bm{\pi}\in R_0$, a set of words $\mathcal{L}_{\bm{\pi}}\subset\Sigma^*$ with $\N$-rational (weighted) growth series that contains unique, geodesic representatives for the $D$-tuples from $C(\bm{\pi})$. Thus the set $\bigcup_{\bm{\pi}\in R_0}\mathcal{L}_{\bm{\pi}}$ contains unique, geodesic representatives for all $\varphi$-twisted conjugacy classes in $G$. To conclude, note that since all $\mathcal{L}_{\bm{\pi}}$ have an $\N$-rational growth series (explicitly computable by Proposition \ref{prop:constructing language}), that $\bigcup_{\bm{\pi}\in R_0}\mathcal{L}_{\bm{\pi}}$ also has an $\N$-rational growth series (simply the finite sum). In particular, the $\varphi$-twisted conjugacy growth series of $G$ with respect to $\Sigma$ is an explicitly computable $\N$-rational function.
\end{proof}

Using the description of a general twisted conjugacy class, given in Lemma \ref{lem:union tc class}, we conclude this section by arguing that the relative growth series of every twisted conjugacy class is rational.

The relative growth function of conjugacy classes and twisted conjugacy classes in virtually abelian groups was shown to be asymptotically equivalent to a polynomial in \cite{DE} and \cite{DL}, respectively, with the latter providing a formula for the degree of polynomial growth.

\begin{definition}[Rational set]
    In a group (or more generally a monoid) $G$, a subset $U$ is called \emph{rational} if there exists a finite subset $A\subset G$ and a regular language $L\subset A^*$ such that the set of elements of $G$ represented by the words in $L$ is exactly $U$.
\end{definition}

For more on rational subsets of virtually abelian groups, we direct the reader to \cite{CE}, \cite{CEL} and \cite{Silva}. In particular, we have the following, which is implicit in \cite{CEL}.

\begin{proposition}
    If $U$ is a rational subset of a finitely generated virtually abelian group then the relative growth series of $U$ is an $\N$-rational function, and can be explicitly computed from a description of the rational set.
\end{proposition}

By Lemma \ref{lem:union tc class}, every twisted conjugacy class in a virtually abelian group is a finite union of cosets of finitely generated subgroups. Such a set is an example of a rational set, the image of a rational expression of the form $\bigcup_{i=1}^n B_i^*a_i$, for finite sets $B_i$. And thus we have the following result.

\begin{theoremLetter}
    Let $G$ be a finitely generated virtually abelian group with a finite weighted (monoid) generating set $\Sigma$ and $\varphi\in \End(G)$. Then the relative growth series of every $\varphi$-twisted conjugacy class in $G$ with respect to $\Sigma$ is an $\N$-rational function. Moreover, given a description for $G$ as in Definition \ref{def:input data}, the $\N$-rational function is explicitly computable.
\end{theoremLetter}
\bibliographystyle{plain}
\bibliography{main}

\begin{thebibliography}{10}

\bibitem{AntolinCiobanu}
Yago Antol\'in and Laura Ciobanu.
\newblock Formal conjugacy growth in acylindrically hyperboic groups.
\newblock {\em Int. Math. Res. Not. IMRN}, (1):121--157, 2017.

\bibitem{Babenko}
Ivan~K. Babenko.
\newblock Closed geodesics, asymptotic volume and the characteristics of growth of groups.
\newblock {\em Izv. Akad. Nauk SSSR Ser. Mat.}, 52(4):675--711, 895, 1988.

\bibitem{Benson}
Max Benson.
\newblock Growth series of finite extensions of {${\bf Z}^{n}$}\ are rational.
\newblock {\em Invent. Math.}, 73(2):251--269, 1983.

\bibitem{Bishop}
Alex Bishop.
\newblock Geodesic growth in virtually abelian groups.
\newblock {\em J. Algebra}, 573:760--786, 2021.

\bibitem{CE}
Laura Ciobanu and Alex Evetts.
\newblock Rational sets in virtually abelian groups: {L}anguages and growth.
\newblock {\em Enseign. Math.}, 71(3-4):231--260, 2025.

\bibitem{CEH}
Laura Ciobanu, Alex Evetts, and Meng-Che Ho.
\newblock The conjugacy growth of the soluble {B}aumslag-{S}olitar groups.
\newblock {\em New York J. Math.}, 26:473--495, 2020.

\bibitem{CEL}
Laura Ciobanu, Alex Evetts, and Alex Levine.
\newblock Effective equation solving, constraints, and growth in virtually abelian groups.
\newblock {\em SIAM J. Appl. Algebra Geom.}, 9(1):235--260, 2025.

\bibitem{CHHR}
Laura Ciobanu, Susan Hermiller, Derek Holt, and Sarah Rees.
\newblock Conjugacy languages in groups.
\newblock {\em Israel J. Math.}, 211(1):311--347, 2016.

\bibitem{DekimpeGoncalves}
Karel Dekimpe and Daciberg Gon\c{c}alves.
\newblock The {$R_\infty$}-property for free groups, free nilpotent groups and free solvable groups.
\newblock {\em Bull. Lond. Math. Soc.}, 46(4):737--746, 2014.

\bibitem{DL}
Karel Dekimpe and Maarten Lathouwers.
\newblock The twisted conjugacy growth of virtually abelian groups.
\newblock {\em J. Algebra}, 676:82--105, 2025.

\bibitem{DE}
Aram Dermenjian and Alex Evetts.
\newblock Conjugacy class growth in virtually abelian groups.
\newblock {\em J. Groups Complex. Cryptol.}, 17(1):Paper No. 1, 14, 2025.

\bibitem{ES}
Samuel Eilenberg and Marcel-Paul Sch\"utzenberger.
\newblock Rational sets in commutative monoids.
\newblock {\em J. Algebra}, 13:173--191, 1969.

\bibitem{Evetts}
Alex Evetts.
\newblock Rational growth in virtually abelian groups.
\newblock {\em Illinois J. Math.}, 63(4):513--549, 2019.

\bibitem{Heisenberg}
Alex Evetts.
\newblock Conjugacy growth in the higher {H}eisenberg groups.
\newblock {\em Glasg. Math. J.}, 65(S1):S148--S169, 2023.

\bibitem{Nrational}
Scott Garrabrant and Igor Pak.
\newblock Counting with irrational tiles, 2014.

\bibitem{GY}
Ilya Gekhtman and Wen-yuan Yang.
\newblock Counting conjugacy classes in groups with contracting elements.
\newblock {\em J. Topol.}, 15(2):620--665, 2022.

\bibitem{GS}
Seymour Ginsburg and Edwin~H. Spanier.
\newblock Semigroups, {P}resburger formulas, and languages.
\newblock {\em Pacific J. Math.}, 16:285--296, 1966.

\bibitem{Greenfeld}
Be'eri Greenfeld.
\newblock Conjugacy growth of free nilpotent groups.
\newblock {\em Internat. J. Algebra Comput.}, 34(6):881--886, 2024.

\bibitem{GubaSapir}
Victor Guba and Mark Sapir.
\newblock On the conjugacy growth functions of groups.
\newblock {\em Illinois J. Math.}, 54(1):301--313, 2010.

\bibitem{HeathKeppelmann}
Philip~R. Heath and Ed~Keppelmann.
\newblock Fibre techniques in {N}ielsen periodic point theory on nil and solvmanifolds. {I}.
\newblock {\em Topology Appl.}, 76(3):217--247, 1997.

\bibitem{Jiang}
Bo~Ju Jiang.
\newblock {\em Lectures on {N}ielsen fixed point theory}, volume~14 of {\em Contemporary Mathematics}.
\newblock American Mathematical Society, Providence, RI, 1983.

\bibitem{Mann}
Avinoam Mann.
\newblock {\em How groups grow}, volume 395 of {\em London Mathematical Society Lecture Note Series}.
\newblock Cambridge University Press, Cambridge, 2012.

\bibitem{Nasybullov}
Timur Nasybullov.
\newblock Reidemeister spectrum of special and general linear groups over some fields contains 1.
\newblock {\em J. Algebra Appl.}, 18(8):1950153, 12, 2019.

\bibitem{SS}
Arto Salomaa and Matti Soittola.
\newblock {\em Automata-theoretic aspects of formal power series}.
\newblock Texts and Monographs in Computer Science. Springer-Verlag, New York-Heidelberg, 1978.

\bibitem{Silva}
Pedro~V. Silva.
\newblock Recognizable subsets of a group: finite extensions and the abelian case.
\newblock {\em Bull. Eur. Assoc. Theor. Comput. Sci. EATCS}, (77):195--215, 2002.

\end{thebibliography}
\end{document}